\documentclass[12pt,a4paper,openany]{article}

\usepackage{amsmath,amssymb,amsfonts,amsthm}
\usepackage{graphics}
\usepackage{mathrsfs}
\usepackage[all]{xy}
\usepackage{color}
\usepackage{bm}

\newtheorem{thm}{Theorem}[section]
\newtheorem{prop}[thm]{Proposition}
\newtheorem{lemma}[thm]{Lemma}

\newtheorem{cor}[thm]{Corollary}

\newtheorem{problem}[thm]{Problem}

\theoremstyle{definition}
\newtheorem{definition}[thm]{Definition}
\newtheorem{rem}[thm]{Remark}
\newtheorem{exam}[thm]{Example}

\newtheorem*{ack}{Acknowledgment}

\newcommand{\sq}{\hfill $\square$}
\newcommand{\kah}{K\"{a}hler}
\newcommand{\dol}{\sqrt{-1}\partial \overline{\partial}}

\newcommand{\hol}{H\"{o}lder}
\newcommand{\ma}{Monge-Amp$\grave{{\rm e}}$re}

\makeatletter
\def\address#1#2{\begingroup
\noindent\parbox[t]{16cm}{%
\small{\scshape\ignorespaces#1}\par\vskip1ex
\noindent\small{\itshape E-mail address}%
\/: #2\par\vskip4ex}\hfill%
\endgroup}%
\makeatother

\makeatletter

\@addtoreset{equation}{section}
\makeatother

\title{Skoda-Zeriahi type integrability and entropy compactness for some measure with $L^1$-density}
\author{Takahiro Aoi}
\date{\today}
\begin{document}
\maketitle

\begin{abstract}
In this paper, we prove the Skoda-Zeriahi type integrability theorem with respect to some measure with $L^1$-density.
In addition, we introduce the log-log threshold in order to detect singularities of {\kah} potentials.
We prove the positivity of the integrability threshold for such a measure and {\kah} potentials with uniform log-log threshold.
As an application, we prove the entropy compactness theorem for a family of potential functions of Poincar\'{e} type {\kah} metrics with uniform log-log threshold.
The Ohsawa-Takegoshi $L^2$-extension theorem and Skoda-Zeriahi's integrability theorem play a very important role in this paper.
\end{abstract}


\tableofcontents

\section{Introduction}

As an important result in complex pluripotential theory, Zeriahi proved the following uniform version of Skoda's integrability theorem.

\begin{thm}{\rm (\cite{Ze}, \cite{Sk}, see \cite[Theorem 2.50]{GZ})}
\label{SZ}
Let $\Omega \subset \mathbb{C}^n$ be an open subset and $\mathcal{U}$ be an $L^1$-compact family in ${\rm PSH}(\Omega)$.
For a compact subset $K \Subset \Omega$, we set $\mathcal{V}(\mathcal{U},K ):= \sup_{x \in K} \{ \nu_\varphi (x)\,\,  | \,\, \varphi \in \mathcal{U}  \}$, where $\nu_\varphi (x)$ denotes the Lelong number of $\varphi$ at $x$.
Then, for any $\alpha \in \left( 0, \frac{2}{\mathcal{V}(\mathcal{U},K)} \right)$, there are an open neighborhood $E$ of $K$ and a constant $C>0$ depending only on $n, K, \mathcal{U}$ and $ \alpha$ such that
\begin{equation}
\label{Skoda}
\int_{E} e^{-\alpha \varphi} d \lambda < C, \,\,\,\,\, \forall \varphi \in \mathcal{U}.
\end{equation}
\end{thm}

Here, $d \lambda$ denotes the Lebesgue measure on $\mathbb{C}^{n}$ and the Lelong number of $\varphi$ at $x$ is defined by $$\nu_\varphi (x):= \sup \left\{ r \geq 0  \, \middle| \,\, \exists C_r > 0 \,\,\, {\rm s.t.} \,\,\, \varphi(z) \leq r \log | z - x|+C_r , \,\, \forall x \in U \right\},$$where $U$ is some (small) neighborhood of $x$.
This theorem has a very important application to the problem of canonical {\kah} metrics on compact complex manifolds, i.e., $\alpha$-invariant introduced by Tian \cite{Ti} (see also \cite[Theorem 4.4.5]{Ho1}).
If we consider the singular case when a measure $m$ has $L^p$-density, i.e., $dm = f d \lambda$ for some $f \in L^p(d\lambda), p>1$, we obtain the similar integrability result by the {\hol} inequality (see {\it tame measures} in the sense of \cite[Definition 1.3]{BBEGZ} and measures defined by {\kah} metrics with {\it conical singularities} along a divisor in \cite[\S 4.4]{Zh}.
For more general cases, see \cite{Ka}).
So, it is a natural problem that we consider the Skoda-Zeriahi type integrability theorem for measures with $L^1$-density (which is not $L^p$-integrable for any $p>1$).
However, we can find easily a bad example, i.e., there exists a plurisubharmonic function with {\bf zero} Lelong number such that the integral in (\ref{SZ}) diverges (see Example \ref{example 1}  and \ref{example 2} in this paper).
So, the Skoda-Zeriahi type integrability is a very subtle problem with respect to measures with $L^1$-density.

In this paper, we prove the Skoda-Zeriahi type integrability theorem for the measure with the $L^1$-density of the form $|z|^{-2} ( \log|z|^{-2})^{\gamma - 2}$ for $\gamma \in [0,1)$ by using the Ohsawa-Takegoshi $L^2$-extension theorem.
In addition, we prove the positivity of the integrability threshold for this measure and a certain family of {\kah} potentials.

Let $\Omega \subset \mathbb{C}^n$ be a bounded pseudoconvex domain such that $|z_1| < e^{-1}$.
We consider the following hyperplane in $\Omega$:
$$
\Omega^\prime:= \Omega \cap \{ z_1 = 0 \}.
$$
We write the Lebesgue measure on $\Omega^\prime$ as $d \lambda^\prime$.
The first result in this paper is the following.

\begin{thm}
\label{L1}
We fix $\gamma \in [0,1)$.
Let $\mathcal{U}$ be an $L^1(d\lambda)$-compact family in ${\rm PSH}(\Omega)$ such that the family $\mathcal{U}^\prime:= \{ \varphi |_{\Omega^\prime} \in {\rm PSH} (\Omega^\prime) \,\, | \,\, \varphi \in \mathcal{U} \}$ is also $L^1(d \lambda^\prime)$-compact.
For any $K \Subset \Omega^\prime$ and $0 < \alpha < \frac{2}{\mathcal{V}(\mathcal{U}^\prime , K)}$, there exists an open neighborhood $V$ of $K$ in $\Omega$ and a constant $C> 0$ depending only on $n, \alpha, \gamma$ and $K$ such that we have
$$
\int_{V}  \frac{e^{-\alpha \varphi } d \lambda}{ |z_1|^2 ( \log |z_1|^{-2})^{2-\gamma } } < C, \,\,\,\,\, \forall \varphi \in \mathcal{U}.
$$
\end{thm}
Note that $ \int_{V} \frac{ d \lambda}{ |z_1|^2 ( \log |z_1|^{-2})^{2-\gamma} }$ is finite unless $\gamma \geq 1$.
In addition, we can easily show that $|z_1|^{-2} ( \log |z_1|^{-2})^{\gamma - 2} \in L^1(d\lambda) \setminus L^p(d\lambda)$ for any $p>1$.

We consider uniform integrability of plurisubharmonic functions to the study of canonical {\kah} metrics of Poincar\'{e} type.
Let $(X, \omega)$ be an $n$-dimensional compact {\kah} manifold and $D$ be a smooth divisor.
Let $s_D \in H^0 (X , \mathscr{O} (D))$ be a defining section of $D$, where $\mathscr{O}(D)$ is the associated line bundle defined by $D$.
We fix a Hermitian metric $h$ on $\mathscr{O}(D)$ and set
$$
\Psi:= - \log \log \Vert s_D \Vert^{-2}_{h} .
$$
A {\kah} metric $\omega_\Phi = \omega + \dol \Phi$ is said to be of Poincar\'{e} type if it satisfies (i) $\omega_\Phi$ is quasi--isometric to $ \omega_{P} := \frac{ \sqrt{-1} dz_1 \wedge d \overline{z}_1}{ |z_1|^2 ( \log |z_1|^{-2})^{2} } + \sqrt{-1} dz_k \wedge d \overline{z}_k $ on any local holomorphic coordinates $(z_1,z_2,...,z_n)$ with $D = \{ z_1 =0 \}$; (ii) $\Phi = O(\Psi) $ as $s_D \to 0$; (iii) the norms of the differentials of $\Phi$ is bounded at any order with respect to $\omega_P$ (see the details in Auvray's paper \cite[Section 1]{Au}).
The top wedge product $\omega_\Phi^n$ defines the measure on $X$ with $L^1 (\omega^n)$-density.
More precisely, we have
$$
\omega_{\Phi}^{n} \approx \frac{1}{\Vert s_D \Vert^{2}_{h} (\log \Vert s_D \Vert^{-2}_{h})^2} \omega^{n}
$$
on a neighborhood of $D$.
By the computation in \cite[\S1.3]{Au}, the measure $\omega_\Phi^n$ has full mass, i.e., $\int_X \omega_\Phi^n = \int_X \omega^n = V$.
(So, we have $\Phi \in \mathcal{E} (X ,\omega)$.)
By the condition (ii) above, we know that $\Phi \to -\infty$ near $D$, so any Poincar\'{e} type {\kah} potential function is unbounded (see \cite{DL}).

We note that the space of $\omega$-plurisubharmonic functions ${\rm PSH} (X,\omega)$ (in particular, the full mass class $\mathcal{E} (X,\omega)$) contains a function $\varphi$ such that $e^{-\alpha \varphi}$ is not integrable with respect to the singular measure defined by a Poincar\'{e} type {\kah} metric for any $\alpha>0$ and $\varphi$ is not a potential function of Poincar\'{e} type {\kah} metrics (see Example \ref{example 2}).
In order to obtain the integrability for potential functions of Poincar\'{e} type {\kah} metrics (with respect to some singular measure), we introduce the following threshold for potential functions (see Remark \ref{reason}).
\begin{definition}[log-log threshold]
Fix an open neighborhood $V$ of $D$.
For $\Phi \in {\rm PSH}(X, \omega ) $, we define the {\it log-log threshold} of $\Phi$ by
\begin{equation*}
\nu_D (\Phi):= \inf \left\{ c > 0 \,\, \middle| \,\,  \exists \varphi \in {\rm PSH} (V, \omega ) \,\, {\rm s.t.}\,\, \varphi |_D \in {\rm PSH}_0(D, \omega|_D ) \,\, {\rm and}\,\,  \Phi \geq c \Psi + \varphi \,\, {\rm on}\,\, V \right\}.
\end{equation*}
For $\mathcal{U} \subset {\rm PSH}(X, \omega )$, we define the log-log threshold of $\mathcal{U}$ by
$
\nu_D (\mathcal{U}):= \sup \{ \nu_D(\Phi) | \Phi \in \mathcal{U} \}.
$
\end{definition}
Here, ${\rm PSH}_0(D, \omega|_D ) :=  \left\{\,  \varphi \in {\rm PSH} (D, \omega |_D) \, \middle| \, \sup_{D} \varphi =0 \, \right\}$ (normalized on $D$ !).
For $\gamma \in [0,1)$, we define the volume form $d \mu_\gamma$ by
$$
d\mu_\gamma:= \frac{1}{\Vert s_D \Vert^{2}_{h} (\log \Vert s_D \Vert^{-2}_{h})^{2-\gamma}} \omega^{n}.
$$
Note that the density function $\Vert s_D \Vert^{-2}_{h} (\log \Vert s_D \Vert^{-2}_{h})^{\gamma -2}$ is in $L^1 (\omega^n)$ but not in $L^p (\omega^n)$ for any $p>1$.
By scaling, we may assume that $d \mu_\gamma$ is a probability measure, i.e., $\int_X d\mu_\gamma = 1$.
By Theorem \ref{SZ} and the proof of Theorem \ref{L1}, we obtain the positivity of the integrability threshold (see Definition \ref{alpha1}), which is the second result:
\begin{thm}
\label{alpha invariant}
Fix $\gamma \in [0,1)$.
Let $\mathcal{U}  \subset {\rm PSH} (X , \omega)$ be an $L^1 (\omega^n)$-compact family such that $\nu_D (\mathcal{U}) < \infty$.
Then, we have $\alpha (\mathcal{U}, d \mu_\gamma) \geq \min \left\{ \frac{2}{ \nu(\mathcal{U}, X)}, \,  \alpha_{\omega |_D} (d \lambda^\prime) , \,  \frac{1-\gamma}{\nu_D (\mathcal{U})} \right\}$.
More precisely, for any $0<\alpha < \min \left\{ \frac{2}{ \nu(\mathcal{U}, X)}, \,  \alpha_{\omega |_D} (d \lambda^\prime) , \,  \frac{1-\gamma}{\nu_D (\mathcal{U})} \right\}$, there exists $C_\alpha > 0$ depending on $\alpha$ such that
$$
\int_{X} e^{-\alpha \Phi  } d \mu_\gamma < C_\alpha , \,\,\,\, \forall \Phi \in \mathcal{U}.
$$
\end{thm}
Here, $\alpha_\omega (\mathcal{U}, d \mu_\gamma )$ denotes the integrability threshold of $\mathcal{U}$ for the singular measure $d \mu_\gamma $ and $ \alpha_{\omega |_D} (d \lambda^\prime)$ denotes the $\alpha$-invariants of $(D , \omega|_D)$ respectively (see Definition \ref{alpha1} and \ref{alpha2} in this paper).
Note that we obtain the trivial upper bound $\alpha(\mathcal{U} , d \mu_\gamma) \leq \alpha_\omega (\mathcal{U} , d \lambda)$ since $d \lambda \leq C d \mu_\gamma$ for some constant $C>0$, where $d \lambda = V^{-1} \omega^n$.

\begin{rem}
This theorem is an analogue of the lower bound for the $\alpha$-invariant of the singular measure denoted by $\mu_{(1-c)D} $ for $c >0$ in \cite[Proposition 6.2]{B} (see also \cite[Proposition 2.2]{DK}).
In \cite{B}, the corresponding log pair $(X, (1-c)D)$ is {\it klt} in the sense of the minimal model program.
In our case, the corresponding pair $(X , D)$ is {\it log canonical} (which is not klt).
\end{rem}

If a family $\mathcal{U} $ is $L^1(\omega^n)$-compact in the full mass class $ \mathcal{E} (X , \omega)$, we have $\nu(\mathcal{U},X)=0$ \cite[Corollary 1.8]{GZ2} (see also \cite[Proposition 2.11]{Da}).
Thus, we immediately have the following corollary.
\begin{cor}
\label{cpt ver}
Fix $\gamma \in [0,1)$.
Assume that $\mathcal{U} $ is $L^1(\omega^n) $-compact in the full mass class $ \mathcal{E} (X , \omega)$ and $\nu_D (\mathcal{U}) < \infty$.
Then, we have $\alpha (\mathcal{U}, d \mu_\gamma) \geq \min \left\{   \alpha_{\omega |_D} (d \lambda^\prime) , \,  \frac{1-\gamma}{\nu_D (\mathcal{U})} \right\}$.

\end{cor}

As one of important applications of the integrability threshold, Chen-Cheng \cite{CC2} proved that the geodesic stability of the (twisted) Mabuchi K-energy is equivalent to the existence of (twisted) constant scalar curvature {\kah} metrics on compact complex manifolds (see the extended K-energy in \cite{BDL1}).
The Chen-Tian formula \cite{Chen1, Ti2} says that the Mabuchi K-energy is decomposed into the sum of the entropy part and the energy part.
The entropy part is called the relative entropy.
A certain type of compactness theorem of the relative entropy (\cite[Theorem 2.17]{BBEGZ}, \cite[Theorem 2.8]{BDL1} and \cite[Theorem 4.44]{Da}) is an important tool of the proof of variational characterization of a constant scalar curvature {\kah} metric \cite{CC2}.
In addition, Zheng \cite{Zh} proved the conic version of Chen-Cheng's result, i.e., the existence of constant scalar curvature {\kah} {\it cone} metrics on compact complex manifolds is equivalent to the geodesic stability of the {\it log} Mabuchi K-energy.
({\kah} cone metrics have mild singularities compared with {\kah} metrics of Poincar\'{e} type.)
In the conic (log) case, the compactness result \cite[Lemma 6.3]{Zh} of the log relative entropy (the log twisted Mabuchi K-energy) for positive cone angle also holds and is an important tool of the proof.
In both cases, the compactness theorem of the (log) relative entropy is proved by the positivity of the integrability threshold (in particular, (log) $\alpha$-invariant).

By Theorem \ref{alpha invariant}, we can show the compactness theorem on the finite energy space $(\mathcal{E}^1 (X ,\omega) ,d_1)$ (\cite[Section 3 and 4]{Da} for details and see \S5.1 in this paper) with respect to Poincar\'{e} type {\kah} metrics.
We define the relative entropy with respect to $\mu_\gamma$ by:
\begin{equation*}
{\rm Ent}_{d\mu_\gamma} \left( V^{-1} \omega_{ \Phi}^{n} \right):= \int_{X} \log \left( \frac{V^{-1} \omega_{ \Phi}^{n}}{d\mu_\gamma} \right) \omega_{\Phi}^{n}
\end{equation*}
The third main result in this paper is as follows:

\begin{thm}
\label{ent cpt}
We take $ \mathcal{U}:= \{ \Phi_j \}_j \subset \mathcal{E}^1 (X,\omega)$.
Assume that $\nu_D ( \mathcal{U}) < \infty$ and there exists $K > 0$ such that $d_1(0, \Phi_j) <K$ and ${\rm Ent}_{d\mu_\gamma} (V^{-1} \omega_{\Phi_j}^{n} ) < K$ for all $j$.
Then, there exists a $d_1$-convergent subsequence of $\{ \Phi_j \}_j$.
\end{thm}

\begin{rem}
By Theorem 2.8 in \cite{BDL1}, the entropy compactness theorem holds for a probability measure $\mu = f \omega^n$ with $f \in L^p (\omega^n)$, i.e., if $| \sup_X \Phi_k | < C$ and ${\rm Ent }_{d\mu} (V^{-1}\omega_{\Phi_k}^{n}) < C$ for all $k$, there exists a $d_1$-convergent subsequence.
In general, if we only assume the existence of a uniform upper bound of $| \sup_X \Phi_k | $ and ${\rm Ent }_{d\mu_0}(V^{-1}\omega_{\Phi_k}^{n})$, then there exists a sequence in $\mathcal{E}^1 (X ,\omega)$ which does not have $d_1$-convergent subsequence (see Example \ref{diverge}).
\end{rem}

In order to prove Theorem \ref{ent cpt}, we need a certain type of local $L^p$-convergence results with respect to the singular measure $d \mu_\gamma$ (see Section 3 in this paper).
By applying Theorem 2 of Di Nezza-Lu \cite{DL}, we have the compactness result for a family in $\mathcal{E}^1 (X , \omega)$ whose {\ma} measure has at most Poincar\'{e} type singularities.

\begin{cor}
\label{DNL}
We take $\{ \Phi_j \}_j \subset \mathcal{E}^1 (X,\omega)$.
Assume that there exists $K > 0$ such that $d_1(0, \Phi_j) <K$  for all $j$.
If $\mathcal{S}(B,1)$-condition holds for some $B>0$ in the sense of \cite{DL} , i.e.,
$$
 \omega_{\Phi_j}^{n}  \leq \frac{B}{ \Vert s_D \Vert^{2}_{h} (\log \Vert s_D \Vert^{-2}_{h})^{2}} \omega^n
$$
for all $j$, then there exists a $d_1$-convergent subsequence of $\{ \Phi_j \}_j$.
\end{cor}

\begin{rem}
If we consider the relative entropy with respect to the Lebesgue measure $d \lambda$ (or the measure defined by a {\kah} metric with conical singularities along $D$), the direct computation tells us that ${\rm Ent}_{d \lambda} (d \mu_0) = + \infty$, i.e., potential functions of Poincar\'{e} type {\kah} metrics do not satisfy the assumption of the entropy finiteness.
So, the compactness results for measures with $L^p$-density (\cite[Theorem 2.17]{BBEGZ}, \cite[Theorem 4.44]{Da} and \cite[Lemma 6.3]{Zh}) cannot be applied to Poincare type {\kah} metrics directly.
On the other hand, by considering the relative entropy with respect to the measure $d\mu_\gamma$, we have ${\rm Ent}_{d \mu_\gamma} (V^{-1} \omega_\Phi^n) < + \infty$ for a Poincar\'{e} type {\kah} metric $\omega_\Phi$ and we can show the similar compactness result under the assumption of uniform log-log threshold.
\end{rem}

At the end of Introduction, we mention the following natural problem.
\begin{problem}
Can we characterize the existence of constant scalar curvature {\kah} metrics of Poincar\'{e} type by certain variational properties of the (log) Mabuchi K-energy?
\end{problem}


This paper is organized as follows.
In Section 2, we prove Theorem \ref{L1}.
The Ohsawa-Takegoshi $L^2$-extension theorem and the integrability theorem of Skoda and Zeriahi (Theorem \ref{SZ}) play an important role in this proof.
In Section 3, we prove that if a sequence of plurisubharmonic functions converges in $L^1(d\lambda)$-topology, it also converges in $L^{p}_{loc} (d\mu_\gamma)$-topology for any $p\geq 1$. 
This is an analogue of the $L^{p}_{loc}(d\lambda)$-convergence results of plurisubharmonic functions with respect to the Lebesgue measure $d \lambda$.
In Section 4, we introduce the log-log threshold for a plurisubharmonic function which is possibly unbounded.
We also prove Theorem \ref{alpha invariant}, i.e., the positivity of the integrability threshold for a family with unform log-log threshold.
In Section 5, we prove Theorem \ref{ent cpt} and Corollary \ref{DNL} by applying Theorem \ref{alpha invariant}.

\begin{ack}
The author would like to thank Natsuo Miyatake for pointing out many typos in the draft of this paper.
He would like to thank Eiji Inoue and Ryoichi Kobayashi for many helpful comments and discussions, and he thanks Yoshinori Hashimoto  for discussions and helping him translate the title and abstract into French.
He also would like to thank Tomoyuki Hisamoto for letting him know that Proposition 2.3 in this paper can be proved by \cite{MV,GuanZhou}.
He is grateful to the anonymous referee for pointing out to improve the previous version of this paper.
This work was supported by JSPS KAKENHI (the Grant-in-Aid for Research Activity Start-up) Grant Number JP23K19020.
\end{ack}


\section{Skoda-Zeriahi type integrability}

Let $\Omega \subset \mathbb{C}^n$ be a bounded pseudoconvex domain with $|z_1| <e^{-1}$.
We consider the hyperplane $\Omega^\prime:= \Omega \cap \{ z_1=0 \}$.
The symbols $d\lambda$ and $d \lambda^\prime$ denote the Lebesgue measures on $\Omega$ and $\Omega^\prime$, respectively.
In this section, we prove Theorem \ref{L1} by applying the (modified) Ohsawa-Takegoshi $L^2$-extension theorem.
More precisely, we consider an $L^2$-extension of the constant function $1$ on $\Omega^\prime$.
We follow B.Y. Chen's simplified proof \cite{Ch} of the Ohsawa-Takegoshi $L^2$-extension theorem and we will modify the weight function.

We will use the following elementary lemma throughout this paper.
\begin{lemma}
For $\gamma \in \mathbb{R}$, we consider the following integral
$$
S_\gamma := \int_{\Omega} \frac{d \lambda  }{|z_1|^2 (\log |z_1|^{-2})^{2 - \gamma}}.
$$
If $\gamma < 1$ then $S_\gamma$ is finite, otherwise $S_\gamma = + \infty$.
\end{lemma}
\begin{proof}
It is enough to compute $S_\gamma$ when $\Omega = B := \{ z \in  \mathbb{C} \,\, |\,\, |z|< e^{-1} \} $ by the Fubini theorem.
By taking the polar coordinates $z = r e^{i \theta}$, we have
$$
\frac{d \lambda  }{|z_1|^2 (\log |z_1|^{-2})^{2 - \gamma}} = \frac{d r d \theta }{r (-2 \log r )^{2 - \gamma}}.
$$
The primitive function of $r^{-1} (-2 \log r )^{ \gamma - 2}$ is $ (2 - 2\gamma )^{-1}(-2 \log r )^{ \gamma - 1}$ if $\gamma \neq 1$ or $-2^{-1}\log (- \log r) $ if $\gamma = 1$.
By taking $r \to 0$, it follows that $S_\gamma < \infty$ only if $\gamma < 1$.
\end{proof}

We recall the Ohsawa-Takegoshi $L^2$-extension theorem.
\begin{thm}[\cite{OT}, see the simplified proof of \cite{Ch}]
Let $\varphi$ be a plurisubharmonic function on $\Omega$.
For any holomorphic function $f$ on $\Omega^\prime = \{z_1 = 0\} \cap \Omega$ with $\int_{\Omega^\prime} |f|^2 e^{-\varphi} d \lambda^\prime < \infty$, there exists a holomorphic function $F$ on $\Omega$ such that $F |_{\Omega^\prime} = f$ and
\begin{equation}
\label{OT}
\int_{\Omega} \frac{|F|^2 e^{-\varphi}}{|z_1|^2 (\log |z_1|^{-2})^2} d \lambda  \leq C \int_{\Omega^\prime} |f|^2 e^{-\varphi} d \lambda^\prime.
\end{equation}
Here, the constant $C>0$ only depends on $n$.
\end{thm}

Note that the measure $|z_1|^{-2} (\log |z_1|^{-2})^{-2} d \lambda$ in the integral in the left hand side of (\ref{OT}) is singular along $\Omega^\prime$ and non-pluripolar, i.e., it does not charge mass on $\Omega^\prime$.
Firstly, we extend the above result for the measure with $L^1 (d\lambda)$-density which is more singular than the  measure in (\ref{OT}).

\begin{prop}{\rm (\cite[Theorem 1.2]{MV}, \cite[Theorem 3.19]{GuanZhou} )}
\label{main prop}
Fix $\gamma \in (0,1)$.
Let $\varphi$ be a plurisubharmonic function on $\Omega$.
For any holomorphic function $f$ on $\Omega^\prime = \{z_1 = 0\} \cap \Omega$ with $\int_{\Omega^\prime} |f|^2 e^{-\varphi} d \lambda^\prime < \infty$, there exists a holomorphic function $F$ on $\Omega$ such that $F |_{\Omega^\prime} = f$ and
$$
\int_{\Omega} \frac{|F|^2 e^{-\varphi}}{|z_1|^2 (\log |z_1|^{-2})^{2 - \gamma}} d \lambda  \leq C \int_{\Omega^\prime} |f|^2 e^{-\varphi} d \lambda^\prime.
$$
Here, the constant $C>0$ only depends on $n$ and $\gamma$.
\end{prop}

\begin{rem}
Indeed, for $\gamma \in (0,1)$, $\frac{  d \lambda }{|z_1|^2 (\log |z_1|^{-2})^{2 - \gamma}} \gg \frac{  d \lambda }{|z_1|^2(\log |z_1|^{-2})^{2 }} $ near $\Omega^\prime$.
\end{rem}

\begin{rem}
Proposition \ref{main prop} has already been proved by \cite{GuanZhou,MV} in a more general framework (see \cite[Definition 1.1 and p.706]{MV}, \cite[Definition 3.18]{GuanZhou}).
By considering a small modification of the weight function in the simplified proof by B.Y.Chen \cite{Ch}, we obtain a short proof of Proposition \ref{main prop} which is based on complex analysis.
(In the proof of \cite{Ch}, the weight function is defined by $\eta_0:= -\rho + \log (-\rho)$ for $\rho < 0$.
Compare $\eta_0$ with the weight function denoted by $\eta = \eta_\gamma$ below.)
The outline of the proof of Proposition \ref{main prop} is similar to B.Y. Chen's proof, but we give a self-contained proof of Proposition \ref{main prop} for reader's convenience.
\end{rem}

{\it Proof of Proposition \ref{main prop}.}
We can take an increasing sequence $\{ \Omega_j \}_j$ where $\Omega_j$ is a bounded smooth pseudoconvex domains such that $\Omega = \cup_j \Omega_j$ and $\Omega_j \Subset \Omega_{j+1}$.
In addition, we can take a smooth approximation $\varphi_k \in {\rm PSH} (\Omega_k) \cap C^\infty (\Omega_k) $ such that $\varphi_k \downarrow \varphi$ on each $\Omega_j$.
So, we may assume that a holomorphic function $f$ can be extended to some neighborhood of $\Omega_j \cap \{z_1 = 0\}$.
In addition, by taking the limit as $k,j \to \infty$, it suffices that we only consider the case when $\varphi$ is smooth and strictly plurisubharmonic on the closure $\overline{\Omega}$.

Fix $\gamma \in (0,1)$.
For sufficiently small $\epsilon >0 $, we set following functions:
$$
\rho:= \log (|z_1|^2 + \epsilon^2),  \,\,\,\, \eta = \eta_\gamma:= - \rho +  (- \rho)^\gamma.
$$
We may assume that $\rho <-1$.
In addition, we set following plurisubharmonic functions:
$$
\phi:=  \varphi + \log |z_1|^2, \,\,\,\, \psi:= - \log \eta.
$$
Let $\chi: \mathbb{R}_{\geq 0} \to \mathbb{R}_{\geq 0}$ be a smooth function such that $\chi \equiv 1 $ on $[0,1/2]$ and $\chi \equiv 0$ on $[1,\infty)$ and $| \dot{\chi}| \leq 3$.
For a holomorphic function $f$ on $\Omega^\prime$, we set
$
f_\epsilon = f \chi (|z_1|^2 / \epsilon^2 ).
$

We consider the $\overline{\partial}$-closed $(0,1)$-form $v_\epsilon:= \overline{\partial} f_\epsilon$.
Since $f$ is holomorphic, we have ${\rm supp}\, v_\epsilon \subset \{ \epsilon^2/2 \leq |z_1|^2 \leq \epsilon^2 \}$.
We can find the minimum solution $u_\epsilon$ such that $
v_\epsilon:= \overline{\partial} u_\epsilon$ and $u_\epsilon \in ({\rm Ker }\overline{\partial})^\perp \subset L^2 (\Omega , e^{- \phi}d \lambda).$
For $\epsilon >0$, the function $\psi$ is bounded, so we have $u_\epsilon e^{\psi} \in ({\rm Ker }\overline{\partial})^\perp \subset L^2 (\Omega , e^{-\psi -\phi} d\lambda ).$
For any $r>0$, the H\"{o}rmander's $L^2$-estimate \cite{Ho2} implies that we can compute as follows.
\begin{eqnarray}
\int_{\Omega} |u_\epsilon |^2 e^{ \psi -\phi}d \lambda \nonumber
&=&  \int_{\Omega} |u_\epsilon  e^{\psi} |^2 e^{-\psi -\phi} d \lambda \\ \nonumber
&\leq& \int_{\Omega} |\overline{\partial} (u_\epsilon e^{\psi}) |_{\dol (\psi + \phi)}^{2} e^{-\psi-\phi} d \lambda \\ \nonumber
&=& \int_{\Omega} |\overline{\partial} u_\epsilon   +  u_\epsilon \overline{\partial} \psi |_{\dol (\psi + \phi) }^{2} e^{\psi-\phi} d \lambda \\ \nonumber
&\leq& (1+r^{-1})  \int_{\Omega} |v_\epsilon |_{\dol (\psi + \phi) }^{2}  e^{\psi-\phi}  d \lambda   +   \int_{\Omega} | u_\epsilon |^2  | \overline{\partial} \psi |_{\dol (\psi + \phi) }^{2} e^{\psi-\phi} d \lambda \\
&& \label{BYC1} \,\,\,\,\, +  r \int_{{\rm supp} v_\epsilon}  | u_\epsilon |^2 |\overline{\partial} \psi |_{\dol (\psi + \phi) }^{2}  e^{\psi-\phi} d \lambda . \
\end{eqnarray}
Here, we have used the Schwarz's inequality.
Note that the term $|\overline{\partial} (u_\epsilon e^{\psi}) |_{\dol (\psi + \phi)}$ makes sense for the plurisubharmonic function $\psi + \phi$ (see \cite{Bl}).
Since $\phi$ is plurisubharmonic, by the direct computation, we obtain
\begin{eqnarray*}
 \dol (\psi + \phi) & \geq & - \frac{ \dol \eta}{\eta} + \frac{\sqrt{-1} \partial \eta \wedge\overline{\partial} \eta}{\eta^{2}}\\
&=&  \left( 1 + \gamma (-\rho)^{\gamma-1}  \right) \frac{ \dol \rho}{\eta}  + \gamma (1-\gamma) \frac{ \sqrt{-1} \partial \rho \wedge\overline{\partial} \rho}{\eta (-\rho)^{2-\gamma} } + \frac{\sqrt{-1} \partial \eta \wedge\overline{\partial} \eta}{\eta^{2}}. 
\end{eqnarray*}
Since $\overline{\partial} \eta  = - \left( 1 + \gamma (-\rho)^{\gamma-1}  \right) \overline{\partial} \rho$, we have
\begin{eqnarray*}
\dol (\psi + \phi) 
&\geq&  \left( \frac{\gamma (1-\gamma)  \eta}{\left( 1 + \gamma (-\rho)^{\gamma-1}  \right)^2 (- \rho)^{2-\gamma}}  +1 \right) \frac{1}{\eta^{2}}  \sqrt{-1} \partial \eta \wedge\overline{\partial} \eta \\
&\geq&  \left( \frac{\gamma (1-\gamma)  \eta}{4 (- \rho)^{2-\gamma}}  +1 \right) \frac{1}{\eta^{2}}  \sqrt{-1} \partial \eta \wedge\overline{\partial} \eta  
\end{eqnarray*}
Here, we have used the inequality $1 + \gamma (-\rho)^{\gamma-1}  \leq 2$.
We can write $\overline{\partial} \psi = - \frac{\overline{\partial} \eta}{\eta} = - \frac{ 1 + \gamma (-\rho)^{\gamma-1}}{\eta} \frac{z_1 d\overline{z}_{1}}{|z_1|^2 + \epsilon^2} .$
Thus, on $\Omega$, we obtain the following estimate:
\begin{eqnarray}
\label{BYC2}
 | \overline{\partial} \psi |_{\dol \phi }^{2} \leq \frac{1}{ \frac{\gamma (1-\gamma)  \eta}{4(- \rho)^{2-\gamma}}  +1  }  <1.
\end{eqnarray}
On the other hand, on ${\rm supp}\, v_\epsilon $, the equality $\dol \rho =  (|z_1|^2 + \epsilon^2 )^{-2} \epsilon^2 dz_1 d\overline{z}_1$  implies that
\begin{equation}
\label{BYC3}
 | \overline{\partial} \psi |_{\dol \phi }^{2} \leq \left( 1 + \gamma (-\rho)^{\gamma-1} \right) \frac{ |z_1|^2}{ \eta \epsilon^2} \leq \frac{4}{\eta}.
\end{equation}
By combining (\ref{BYC1}), (\ref{BYC2}) and (\ref{BYC3}), we obtain the following estimate.
\begin{eqnarray}
\label{BYC4}
\int_{\Omega} \left( 1 - \frac{1}{  \frac{\gamma (1-\gamma)  \eta}{4(- \rho)^{2-\gamma}}  +1  }  -\frac{4r}{\eta} \right) |u_\epsilon |^2 e^{ \psi -\phi} d \lambda \leq  (1+r^{-1})  \int_{\Omega} |v_\epsilon |_{\dol (\psi  + \phi) }^{2}  e^{\psi -\phi}  d \lambda .
\end{eqnarray}
For sufficiently small $\epsilon >0$, we have $\eta \approx -\rho$.
So, by taking sufficiently small $\epsilon > 0$ and $r>0$, we obtain
$$
1 - \frac{1}{  \frac{\gamma (1-\gamma)  \eta}{4(- \rho)^{2-\gamma}}  +1  }  -\frac{4r}{\eta} \approx \frac{1}{\eta^{1-\gamma}}.
$$
Therefore, the left hand side in (\ref{BYC4}) can be bounded from below by $\int_\Omega \eta^{\gamma-1} |u_\epsilon |^2 e^{ \psi -\phi} d \lambda$.

Next, we consider the limit of the right hand side in (\ref{BYC4}) as $\epsilon \to 0$.
Since $f$ is holomorphic, we have $v_\epsilon = \overline{\partial} f_\epsilon = f \chi^\prime (|z_1|^2/\epsilon^2) \frac{z_1 d\overline{z}_1}{\epsilon^2}.$
On ${\rm supp}\, v_\epsilon  \subset \{ \epsilon^2/2 \leq |z_1|^2 \leq \epsilon^2 \}$, we have
$$
|v_\epsilon |_{\dol (\psi  + \phi) }^{2}  e^{\psi -\phi} \leq \frac{|f|^2 |\dot{\chi}|^2 |z_1|^2}{\epsilon^4} \frac{(|z_1|^2 + \epsilon^2)^2 \eta}{\epsilon^2} e^{\psi -\phi} \leq 36 |f|^2 e^{ -\phi}.
$$
Here, we have used the fact that $\psi = - \log \eta$ and $|\dot{\chi}|\leq 3$.
By scaling $\epsilon w:= z_1$, the Fubini theorem implies that
\begin{eqnarray*}
\limsup_{\epsilon \to 0} \int_{\Omega} |v_\epsilon |_{\dol (\psi  + \phi) }^{2}  e^{\psi -\phi}d \lambda \leq   C_1 \int_{\Omega^\prime}|f|^2 e^{-\varphi} d \lambda^\prime
\end{eqnarray*}
as $\epsilon \to 0$.
Here, we set $C_1:= 36 \int_{\{ 1/2 < |w|^2 < 1\}}  |w|^{-2} i dw\wedge d\overline{w}$ and this constant $C_1$ is independent of $f$ and $\varphi$.

We set a holomorphic function $F_\epsilon $ by $F_\epsilon:= f_\epsilon - u_\epsilon$.
Note that $u_\epsilon$ vanishes on $\Omega^\prime$.
By taking $\epsilon \to 0$, we obtain a holomorphic function $F:= \lim_{\epsilon \to 0} F_\epsilon$ on $\Omega$ such that $F |_{\Omega^\prime} = f$ and
$$
 \int_{\Omega} \frac{  |F|^2 e^{-\varphi}d \lambda}{|z_1|^2 (- \log |z_1|^2 )^{2 - \gamma}} \leq C_1 \int_{\Omega^\prime} |f|^2 e^{-\varphi} d \lambda^\prime .
$$\sq

\begin{rem}
From the proof above, we can easily show that $C \propto \gamma^{-1} (1-\gamma)^{-1}$.
\end{rem}

Note that $\mathcal{V} (\mathcal{U}^\prime , K)$ is finite since the Lelong number $\nu (\varphi , x)$ is upper semicontinuous on ${\rm PSH}(\Omega^\prime) \times \Omega^\prime$ (see \cite[Exercise 2.7]{GZ}).
By using Proposition \ref{main prop}, we can prove Theorem \ref{L1}.

\medskip

{\it Proof of Theorem \ref{L1}.}
We apply Proposition \ref{main prop} for a constant function $f\equiv1$ on $\Omega^\prime$.
We show that there exists an open neighborhood $V$ of $\Omega^\prime$ in $\Omega$ such that the norm of the extension satisfies $|F|> 1/2$ on $V$.
We take $K \Subset \Omega^\prime$ and $0< \alpha < \frac{2}{\mathcal{V} (\mathcal{U}^\prime,K)}$.
By assumption, Theorem \ref{SZ} implies that there exists an open neighborhood $E$ of $K$ in $\Omega^\prime$ and $C_\alpha>0$ such that $\int_{E} e^{-  \alpha \varphi} d \lambda^\prime < C_\alpha $ for any $\varphi \in \mathcal{U}$.
We consider a finite cover of $K$ by small balls, so we may assume that $E$ is a ball in $\Omega^\prime$.
In addition, we take a bounded Stein neighborhood $\tilde{E}$ of $K$.
Since the family $\mathcal{U}$ is $L^1(d\lambda)$-compact, we may assume that $\varphi \leq 0$ on $\tilde{E}$.
By using Proposition \ref{main prop}, we have
$$
 \int_{\tilde{E}}  |F|^2d \lambda  \leq  \int_{\tilde{E}} \frac{  |F|^2 e^{-\alpha \varphi}d \lambda}{|z_1|^2 (- \log |z_1|^2 )^{2 - \gamma }} \leq \int_{E} e^{-  \alpha \varphi} d \lambda^\prime < C_\alpha.
$$
By the mean value inequality with respect to $|F|^2$, we obtain the $C^0$-bound of $|F|^2$ on $V_r:= \{  z \in \Omega \, | \, d(z , \partial \tilde{E})>r \}$ for sufficiently small $r > 0$.
Cauchy integral formula tells us that we obtain the uniform $C^1$-bound of $F$ on $V_{2r}$, i.e., $| \partial F | < C^\prime$ for some constant $C^\prime>0$ on $V_{2r}$.
By using the fact that $F |_{\Omega^\prime} = 1$, we obtain the inequality $|F(q)| \geq |F(p)| - |F(p) - F (q)| \geq 1 - C^\prime |p-q|$ for $p\in \Omega^\prime$ and $q \in V_{2r}$.
Thus, we can find a sufficiently small neighborhood $V$ of $\Omega^\prime$ such that $|F| >1/2$ on $V$, so we have
$$
 \int_{V} \frac{ e^{-\alpha \varphi}d \lambda}{|z_1|^2 (- \log |z_1|^2 )^{2 - \gamma}} \leq 4 \int_{E} e^{-\alpha \varphi} d \lambda^\prime <4C_\alpha.
$$\sq

\begin{rem}
Demailly-Koll\'{a}r and Berman use the Ohsawa-Takegoshi $L^2$-extension theorem in order to show the integrability of $e^{-\alpha \varphi}$ (\cite[Proposition 2.2]{DK} and \cite[Proposition 6.2]{B}).
In our case, we want to deal with potential functions of Poincar\'{e} type {\kah} metrics, i.e., such functions must be equal to $- \infty$ on $D$ (equivalently, on $\Omega^\prime$), so we need to the $L^2$-extension result for the singular measure $d \mu_\gamma$ (see Section 4 in this paper).
\end{rem}


\section{Local $L^p(d \mu_\gamma)$-convergence}

In this section, for $p \geq 1$, we show the local $L^p (d \mu_\gamma)$-convergence of plurisubharmonic functions with some integrability condition.
For simplicity, we use the same symbol as follows:
$$
d\mu_\gamma := \frac{d \lambda}{|z_1|^2 ( \log |z_1|^{-2} )^{2-\gamma}}.
$$
The following theorem is an analogue of the fact that if a sequence of plurisubharmonic functions on $\Omega$ converges in the sense of $L^1 (d\lambda)$-topology then it converges in the sense of $L^{p}_{loc} (d\lambda)$-topology for any $p\geq1$ (see \cite[Theorem 1.48]{GZ}).

\begin{thm}
\label{local Lp}
Assume that $\varphi_j \in {\rm PSH}(\Omega)$ converges to $\varphi \in {\rm PSH}(\Omega)$ in $L^1 (d \lambda)$-topology.
We also assume that there exists $\alpha>0$ such that for any open subset $E \Subset \Omega$, there exists $C_\alpha> 0$ such that
$\int_{E} e^{-\alpha\varphi_j} d \mu_\gamma <C_\alpha$ for any $j$.
Then $\varphi_j $ converges to $\varphi $ in $L^{p}_{loc} (d \mu_\gamma)$-topology for any $p\geq 1$.
\end{thm}

\begin{proof}

By the result of $L^{p}_{loc}$-convergence with respect to the Lebesgue measure $d \lambda$ (see \cite[Theorem 1.48]{GZ}), it suffices to show the $L^{p}_{loc} (d \mu_\gamma)$-convergence on arbitrary small open ball which intersects $\Omega^\prime$.
Since $\varphi_j \to \varphi$ in $L^1 (d \lambda)$, we can find a subsequence $\varphi_{j_k}$ such that $\varphi_{j_k} \to \varphi$ almost everywhere as $k \to \infty$.
In addition, we may assume that $\varphi_j \leq 0$ by taking a sufficiently small ball $B$ in $\Omega$.
Fatou's lemma implies that 
$$
\int_B e^{-\alpha \varphi} d \mu_\gamma \leq  \liminf_k \int_B e^{-\alpha \varphi_{j_k}} d \mu_\gamma \leq C_\alpha
$$ 
and we obtain $ e^{-\varphi} \in L^{\alpha}_{loc} (d \mu_\gamma)$.
By the elementary inequality $\frac{(-\alpha \varphi_j)^p}{p!} \leq e^{-\alpha\varphi_j} $ for any $p\geq 1$, we have $ \varphi_j , \varphi \in L^{p}_{loc} (d\mu_\gamma)$.

For $\epsilon \in (0,1)$, we define an convex and increasing function by
$$
\chi (s) := \int_0^s (\log (x+1))^\epsilon dx , \,\,\,  s \in \mathbb{R}_{\geq 0} .
$$
Note that $\chi (s) \leq (s+1) (\log (s+1))^\epsilon +O(1) \leq  s (\log s)^\epsilon +O(1)$.
Recall that the Legendre transformation $\chi^*(t)$ is defined by $\chi^*(t) := \sup_{s \geq 0} ( st - \chi(s))$.
By the direct computation, we have
$$
\chi^*(t) = \int_0^t ( \exp (y^{1/\epsilon}) - 1) dy  ,\,\,\,  t \in \mathbb{R}_{\geq 0}.
$$
Note that $\chi^*(t) \leq t   \exp (t^{1/\epsilon})$ for any $t\geq 0$.
For simplicity, we write $d \mu_\gamma = f_\gamma d \lambda$.
If $\gamma + \epsilon <1$, it follows from the direct computation that $C_{\gamma, \epsilon} := \int_\Omega \chi (f_\gamma) d\lambda <\infty$.
Indeed, we have
$$
\chi (f_\gamma)  \approx \frac{1}{|z_1|^2 \left( \log |z_1|^{-2} \right )^{2-\gamma - \epsilon}}
$$
near $\{z_1 = 0\} = \Omega^\prime$.
We fix $r < \alpha/3$.
Note that we may assume that $d\lambda$ is a probability measure on $B$ by scaling.
By the {\hol} inequality for the pair $\chi$ and $\chi^*$ (see \cite[Proposition 2.15]{BBEGZ} and \cite[Proposition 1.3]{Da}), we have
\begin{eqnarray*}
\int_B | \varphi_j - \varphi |^\epsilon d\mu_\gamma 
&\leq& 2\Vert  r | \varphi_j - \varphi |^\epsilon \Vert_{L^\chi(d \lambda)} \Vert r^{-1} f_\gamma \Vert_{L^{\chi^*}(d \lambda)}.
\end{eqnarray*}
Here, the symbols $\Vert \cdot \Vert_{L^\chi(d \lambda)}$ and $\Vert \cdot \Vert_{L^{\chi^*} (d \lambda)}$ are the Orlicz norms with respect to the functions $\chi$ and $\chi^*$, respectively (see \cite[Section 2]{BBEGZ} and \cite[Section 1]{Da}).
Since $\chi^*(t) \leq t   \exp (t^{1/\epsilon})$, we have
\begin{eqnarray*}
\int_B \chi^* (r| \varphi_j - \varphi |^\epsilon) d\lambda 
&\leq& \int_B r | \varphi_j - \varphi |^\epsilon e^{ r| \varphi_j - \varphi |} d \lambda \\
& \leq &  r \int_B  | \varphi_j - \varphi |^\epsilon e^{ -r \varphi_j}  e^{ -r \varphi} d \lambda \\
& \leq &  r \Vert  | \varphi_j - \varphi |^\epsilon \Vert_{L^3 (d\lambda)}  \Vert e^{ -r \varphi_j}   \Vert_{L^3 (d\lambda)} \Vert  e^{ -r \varphi}  \Vert_{L^3 (d\lambda)}.
\end{eqnarray*}
Here, we have used the {\hol} inequality in the last inequality.
It follows from \cite[Theorem 1.48]{GZ} that $\Vert  | \varphi_j - \varphi |^\epsilon \Vert_{L^3 (d\lambda)} \to 0$.
Thus, the assumption of Theorem \ref{local Lp} implies that we have $\int_B | \varphi_j - \varphi |^\epsilon d\mu_\gamma  \to 0$ as $j \to \infty$.
By the Schwarz's inequality (with respect to the measure $d\mu_\gamma$), we have
\begin{eqnarray*}
\int_B | \varphi_j - \varphi |^p d\mu_\gamma 
&\leq& \left( \int_B | \varphi_j - \varphi |^\epsilon d\mu_\gamma \right)^{1/2} \left(  \int_B | \varphi_j - \varphi |^{2p -\epsilon} d\mu_\gamma \right)^{1/2}
\end{eqnarray*}
for any $p \geq 1$.
We have already showed that the last integral in the right hand side is bounded, so we can prove that $\int_B | \varphi_j - \varphi |^p d\mu_\gamma \to 0$ as $j \to \infty$.
Thus, we have finished the proof.
\end{proof}

\begin{rem}
In general, a sequence of plurisubharmonic functions which converges to some plurisubharmonic function in $L^1 (d\lambda)$, does not converges in $L^p (d\mu_\gamma)$.
Indeed, for $t \in (0,1)$, $u_t:= \log (|z|^2+t)\in {\rm PSH}(\mathbb{B})$ converges to $u_0:= \log |z|^2 \in {\rm PSH}(\mathbb{B})$ in $L^1 (d\lambda)$ as $t \to 0$.
Here, $\mathbb{B}$ denotes the unit disk in $\mathbb{C}$.
But, $u_t$ does not converge to $u_0$ in $L^p (d\mu_0)$ for any $p \geq 1$ because $u_0 \notin L^p (d\mu_0)$.
\end{rem}

By Theorem \ref{L1}, we immediately have the following corollary.
\begin{cor}
Assume that $\varphi_j \in {\rm PSH}(\Omega)$ converges to $\varphi \in {\rm PSH}(\Omega)$ in $L^1 (d \lambda)$-topology.
If the family $\{ \varphi_j |_{\Omega^\prime} \}$ is $L^1 (d\lambda^\prime)$-compact, then $\varphi_j $ converges to $\varphi $ in $L^{p}_{loc} (d \mu_\gamma)$-topology for any $p\geq 1$.
\end{cor}


\section{Integrability threshold}

In this section, we discuss the positivity of the integrability threshold with respect to the measure $d \mu_\gamma$ for some $L^1$-compact family of quasi--plurisubharmonic functions.

\subsection{Nonpluripolar product and full mass }

Let $(X ,\omega)$ be an $n$-dimensional compact {\kah} manifold.
We recall the definitions of the nonpluripolar complex {\ma} measure and the space of $\omega$-plurisubharmonic functions with full mass.
We define the space of $\omega$-(quasi-)plurisubharmonic functions by
$$
{\rm PSH}(X , \omega):= \left\{ \, \varphi: X \to \mathbb{R} \cup \{-\infty \} \,   \,  {\rm usc \, }\middle|\,  \omega_\varphi:= \omega + \dol \varphi \geq 0 \,\,\, ({\rm\, as \,\, a \,\, current \, })  \right\}.
$$
Here, the word ``usc'' means ``upper semicontinuous''.
For $\varphi_k \in {\rm PSH}(X , \omega) \cap L^\infty (X)$  $(k=1,...,r)$, Bedford-Taylor \cite{BT} showed that the wedge product $ \wedge_{k=1}^r \omega_{\varphi_k}$ is an well-defined closed positive $(r,r)$-current (see also \cite[\S3 in Chapter III]{De}).
So, we can define the complex {\ma} measure $\omega_\varphi^n$ for $\varphi \in  {\rm PSH}(X , \omega) \cap L^\infty (X)$.
Moreover, for a possibly unbounded function $\varphi \in  {\rm PSH}(X , \omega)$, we can define the complex {\ma} measure as follows.

\begin{definition}{\rm (\cite{GZ2, BEGZ})}
\label{nnpl}
For $j \in \mathbb{Z}_{>0}$, we define the canonical cut-offs of $u \in {\rm PSH}(X ,\omega)$ by $\varphi_j:= \max \{ \varphi , -j \} \in {\rm PSH}(X,\omega) \cap L^\infty (X).$
We define the nonpluripolar complex {\ma} measure of $\varphi \in {\rm PSH}(X, \omega)$ by
$$
\omega_{\varphi}^n:= \lim_{j \to \infty}\bm{1}_{ \{ \varphi >-j \}} \omega_{\varphi_j}^n .
$$
\end{definition}

By using Definition \ref{nnpl}, we define the full mass class in $ {\rm PSH}(X, \omega)$ as follows.
\begin{definition}{\rm (\cite[Definition 2.1]{BEGZ} )}
\label{full mass}
The set of $\omega $-plurisubharmonic functions of {\it full mass} is defined by
$$
\mathcal{E} (X, \omega):= \left\{  \varphi  \in {\rm PSH}(X,\omega) \,\,   \middle| \,\, \int_X \omega_{\varphi}^{n} =  V  \right\}.
$$
\end{definition}
Note that the complex {\ma} measure $\omega_{\varphi_j}^n$ does not put mass on pluripolar sets, i.e.,
$
\int_X \omega_{\varphi_j}^{n} = \int_X \omega^n = V
$
for any $j$ by the Stokes theorem.
In general, for $\varphi \in {\rm PSH}(X, \omega)$, the total {\ma} mass of $\omega_{\varphi}^{n}$ is equal to or less than $V$, i.e., $\int_X \omega_\varphi \leq V$.
Full mass $\omega$-plurisubharmonic functions have mild singularities in the following sense.
\begin{prop}{\rm (\cite[Corollary 1.8]{GZ2} and see \cite[Proposition 2.11]{Da})}
\label{zero Lelong}
For $\varphi \in \mathcal{E} (X ,\omega)$, the Lelong number of $\varphi$ is identically zero.
\end{prop}

\begin{exam}
The computation in Auvray's paper \cite[\S 1.3]{Au} implies that $\Psi = - \log \log \Vert s_D \Vert_{h}^{-2} \in \mathcal{E} (X, \omega)$, where $s_D$ is a defining section of  a smooth divisor $D$ and $h$ is a Hermitian metric on $\mathscr{O}(D)$ (see the next subsection).
So, we obtain ${\rm PSH} (X, \omega) \cap L^\infty (X) \subsetneq \mathcal{E} (X, \omega)$.
For more examples of unbounded $\omega$-plurisubharmonic functions with full mass, see the systematic construction in \cite[Example 2.14]{GZ2} and \cite[\S2.3]{Da}.

\end{exam}

\subsection{Positivity of integrability threshold}

Let $(X, \omega)$ be an $n$-dimensional compact {\kah} manifold.
Let $D$ be a smooth divisor on $X$ and $s_D$ be a defining section of $D$.
We take a Hermitian metric $h$ on the line bundle $\mathscr{O}(D)$ which is the associated line bundle of $D$.
By scaling, we may assume that $\Vert s_D \Vert_{h} < e^{-1}$. 
Set $\Psi:= - \log \log \Vert s_D \Vert^{-2}_{h}$.
By scaling again, we may assume that $\Psi \in {\rm PSH} (X ,\omega)$ (see \cite{Au}).
For $\gamma \in [0,1)$, we set the singular volume form $d\mu_\gamma$ with $L^1(\omega^n)$-density as follows:
$$
d \mu_\gamma:=  \frac{\omega^{n}}{ \Vert s_D \Vert^{2}_{h} (\log \Vert s_D \Vert^{-2}_{h})^{2-\gamma} }  .
$$
If $\gamma =0$, we have $d \mu_0 \approx \omega_{\Psi}^{n}$ by the direct computation.
Since $d \mu_\gamma $ has finite volume for $\gamma \in [0,1)$, we may assume that $d \mu_\gamma$ is a probability measure on $X$ by scaling.

Firstly, we define the integrability threshold for {\it a family} in the set of quasi-plurisubharmonic functions and a probability measure. 

\begin{definition}
\label{alpha1}
For $\mathcal{U} \subset {\rm PSH}(X, \omega)$ and a probability measure $d m$ on $X$, we define the {\it integrability threshold of $\mathcal{U}$ for $d m$} by
\begin{equation}
\label{integral3}
\alpha (\mathcal{U}, d m):= \sup \left\{  \alpha \geq 0 \,\, \middle| \,\  \exists C_\alpha >0 \,\, {\rm s.t.} \,\, \int_{X} e^{-\alpha \Phi  } d m<C_\alpha  \,\,\, {\rm for} \,\,\, \forall \Phi \in \mathcal{U} \right\}.
\end{equation}

\end{definition}

Secondly, we recall the definition of the $\alpha$-invariant introduced by Tian \cite{Ti} which is a holomorphic invariant of a {\kah} manifold.

\begin{definition}{\rm (\cite[p.229]{Ti})}
\label{alpha2}
We denotes the set of normalized quasi--plurisubharmonic functions as ${\rm PSH}_0 (M , \theta)$ for a {\kah} manifold $ (M , \theta)$, i.e.,
$$
{\rm PSH}_0 (M , \theta) :=  \left\{\,  \varphi \in {\rm PSH}(M, \theta) \, \middle| \, \sup_{M} \varphi =0 \, \right\}.
$$
We write $d\lambda = V^{-1} \omega^n$ and $d \lambda^\prime = V_{D}^{-1} \omega^{n-1} |_D$ respectively, where $V_D:= \int_D \omega^{n-1}$.
We define the {\it $\alpha$-invariants for $(X, \omega)$ and $(D, \omega |_D)$} by:
$$\alpha_\omega (d \lambda):= \alpha ( {\rm PSH}_{0} (X , \omega) , d \lambda) , \,\,\,\,\, \alpha_{\omega|_D} (d \lambda^\prime):= \alpha ( {\rm PSH}_{0} (D , \omega |_D) , d \lambda^\prime ),$$
respectively.

\end{definition}

Note that the class ${\rm PSH}_0(X, \omega )$ is $L^1 (d \lambda)$-compact (\cite[Proposition 1.7]{GZ1}, \cite[\S 1.3]{BBEGZ}).
By the boundedness of the Lelong numbers on compact complex manifold (see \cite[Proposition 2.1]{Ti}, \cite[Lemma 2.3.43]{MM}), it is showed that the $\alpha$-invariant of the set of normalized quasi-plurisubharmonic functions for the Lebesgue measure is positive.

\begin{thm}{\rm (\cite[Theorem 1.5]{Ti} and \cite[Proposition 1.4]{BBEGZ})}
\label{Tian}
We have $\alpha_\omega (d \lambda) > 0$ $($so, $ \alpha_{\omega|_D} (d \lambda^\prime) > 0$$)$.
\end{thm}

\begin{rem}
Note that $a_\omega (d \lambda) $ in Definition \ref{alpha2} does not depends on the choice of (smooth) {\kah} metric in the fixed cohomology class $[\omega]$.
It is well-known that $\alpha ( d \lambda)$ is invariant under the action of holomorphic automorphisms (see \cite[Proposition 2.1]{Ti}).
\end{rem}

\begin{rem}
In the Fano case, i.e., $[\omega] = c_1 (X)$, Tian proved that there exists a {\kah}-Einstein metric if $\alpha_\omega (d \lambda) > \frac{n}{n+1}$, i.e., the positivity condition for $\alpha$-invariant gives a sufficient condition of the existence of a {\kah}-Einstein metric \cite[Theorem 2.1]{Ti}.
The similar results for a log Fano pair $(X,D)$ and a general polarization are given in \cite[\S 4.4]{BBEGZ}, \cite[Theorem 1.4]{Der} and \cite[Proposition 4.22]{Zh}.
\end{rem}

\begin{rem}
We know that the positivity of $\alpha$-invariant is related to the Lelong number of (quasi-)plurisubharmonic functions and singularities of a probability measure.
If we consider a measure with $L^p (d\lambda)$-density for some $p>1$, i.e., $dm = f d \lambda , \,\, f \in L^p (d\lambda)$, Theorem \ref{Tian} and the {\hol} inequality tell us that the integrability threshold for the set of normalized $\omega$-plurisubharmonic functions ${\rm PSH}_0 (X , \omega)$ is positive (see \cite[Proposition 1.4]{BBEGZ} and \cite[\S4.4]{Zh}).
However, in our case, i.e., the density of the measure $d\mu_\gamma$ is not $L^p (d\lambda)$-integrable for any $p>1$, we can find an $\omega$-plurisubharmonic function with {\bf zero} Lelong number such that the integral in (\ref{integral3}) diverges.
The following examples tell us that we need to restrict the function space in order to obtain the positivity of the integrability threshold for $d \mu_\gamma$.
\end{rem}

\begin{exam}[integrability threshold for $\mathcal{U} \ni \Psi$]
\label{example 1}
If a family $\mathcal{U}$ contains a potential function of Poincar\'{e} type {\kah} metric, then we have $\alpha (\mathcal{U}, d \mu_\gamma) \neq \infty$.
For instance, if $\Psi =  - \log \log \Vert s_D \Vert^{-2}_{h} \in \mathcal{E} (X, \omega) \cap \mathcal{U}$, we obtain $\alpha (\mathcal{U}, d \mu_\gamma) \leq 1 -\gamma$ by the following computation:
$$
\int_{X} e^{- (1-\gamma) \Psi } d \mu_\gamma = \int_{X}  \frac{\omega^{n}}{ \Vert s_D \Vert^{2}_{h} (\log \Vert s_D \Vert^{-2}_{h}) } = \infty.
$$
\end{exam}

\begin{exam}[vanishing case]
\label{example 2}
For sufficiently small $\epsilon > 0$, we set $\varphi_\epsilon:= -\epsilon (\log \Vert s_D \Vert^{-2}_{h})^\epsilon \in \mathcal{E} (X, \omega)$.
The Lelong number of $\varphi_\epsilon$ is $0$ on $X$.
For any $\alpha>0$, we can find a sufficiently small neighborhood $V$ of $D$ such that $\alpha \varphi_\epsilon \leq (1-\gamma) \Psi$ on $V$, so we have
\begin{eqnarray*}
 \int_{X} e^{- \alpha \varphi_\epsilon } d \mu_\gamma 
&\geq&   \int_{V}    \frac{e^{-(1-\gamma)\Psi}\omega^{n}}{ \Vert s_D \Vert^{2}_{h} (\log \Vert s_D \Vert^{-2}_{h})^{2-\gamma} } =  \infty.
\end{eqnarray*}
Thus, if $\varphi_\epsilon \in \mathcal{U}$, we have $\alpha (\mathcal{U}, d \mu_\gamma) = 0$.
\end{exam}

By Example \ref{example 2}, we know that the integrability threshold for the singular measure $d \mu_\gamma$ must be 0 for {\bf all} (normalized) $\omega $-plurisubharmonic functions with full mass.
But, from Example \ref{example 1}, we expect that the integrability threshold is positive for a family of {\kah} potentials of Poincar\'{e} type.
In order to consider elements in $\mathcal{E}(X ,\omega)$ which make the corresponding integrability threshold positive, we introduce the following threshold for potential functions.
\begin{definition}[log-log threshold]
\label{log-log}
Fix a sufficiently small open neighborhood $V$ of $D$.
For $\Phi \in {\rm PSH}(X, \omega ) $, we define the {\it log-log threshold of $\Phi$} by
\begin{equation*}
\nu_D (\Phi):= \inf \left\{ c > 0 \,\, \middle| \,\,  \exists \varphi \in {\rm PSH} (V, \omega ) \,\, {\rm s.t.}\,\, \varphi |_D \in {\rm PSH}_0 (D , \omega_D) \,\, {\rm and}\,\,  \Phi \geq c \Psi + \varphi \,\, {\rm on}\,\, V \right\}.
\end{equation*}
For $\mathcal{U} \subset {\rm PSH}(X, \omega )$, we define the log-log threshold of $\mathcal{U}$ by
$
\nu_D (\mathcal{U}):= \sup \{ \nu_D(\Phi) | \Phi \in \mathcal{U} \}.
$
\end{definition}

\begin{rem}
\label{reason}
For $\Psi \in {\rm PSH}(X ,\omega)$ in Example \ref{example 1}, we have $\nu_D (\Psi) = 1$. 
On the other hand, for $\varphi_\epsilon \in {\rm PSH}(X ,\omega)$ in Example \ref{example 2}, we have $\nu_D (\varphi_\epsilon) = + \infty$ for any $\epsilon >0$ since the inequality $\varphi_\epsilon \leq c \Psi$ holds on a sufficiently small neighborhood of $D$ for any $c>0$.
Thus, the subset $\{ \nu_{D} (\Phi) <  + \infty \} \subset {\rm PSH }(X , \omega)$ doesn't contain functions like $\varphi_\epsilon$.
\end{rem}

\begin{rem}
For $\Phi \in {\rm PSH}_0 (X , \omega)$, if the volume form $\omega_\Phi^n$ satisfies $\mathcal{S} (B,1)$-condition for some $B>0$ in the sense of \cite{DL}, we can easily show that there exists $C>0$ depending only on $B$ such that $\nu_D (\Phi) < C$.
However, the converse does not hold in general.
For instance, by assuming that $X$ is projective, we can find a smooth hypersurface $F \subset X$ transverse to $D$ which is defined by $\sigma \in H^0 (X , \mathscr{O}(F))$.
We take a smooth Hermitian metric $h_F$ on $\mathscr{O}(F)$. 
For sufficiently small $\epsilon >0$, we can construct an $\omega$-plurisubharmonic function $\Psi_F := - \epsilon ( \log \Vert \sigma \Vert^{-2}_{h_F})^{\epsilon}$ such that $\nu_D (\Psi_F) < \infty$ but $\omega_{\Psi_F}^{n}$ does not satisfies $\mathcal{S}(B,1)$-condition for any $B>0$.
\end{rem}

For a family $\mathcal{U} \subset \mathcal{E} (X ,\omega)$ with uniform log-log threshold, we prove Theorem \ref{alpha invariant}, i.e., the positivity of the integrability threshold.

\medskip

{\it Proof of Theorem \ref{alpha invariant}.}
The proof of the case when $\gamma \in (0,1)$ is similar to the case when $\gamma =0$.
So, we only prove Theorem \ref{alpha invariant} in the case $\gamma = 0$.
Let $W \Subset V $ be a sufficiently small open neighborhood  of $D$ specified later.
Here $V$ is the fixed open neighborhood of $D$ in Definition \ref{log-log}.
We write
$$A:= \min \left\{ \frac{2}{\nu (\mathcal{U}, X)}, \,  \alpha_{\omega |_D} (d \lambda^\prime) , \,  \frac{1}{\nu_D (\mathcal{U})} \right\}.
$$
Fix $0 < \alpha < A$.
On $X\setminus W$, Skoda-Zeriahi's integrability theorem (Theorem \ref{SZ}) and the $L^1$-compactness of $\mathcal{U}$ imply that there exists $C_\alpha >0$ such that $\int_{X \setminus W} e^{-\alpha \Phi} d \mu_0 <C_\alpha$ for any $\Phi \in \mathcal{U}$.
Thus, it is enough to show that $\int_{W} e^{-\alpha \Phi}d \mu_0$ is uniformly bounded.
We also fix a sufficiently small constant $\epsilon>0$ such that $ \alpha + \epsilon < A$.
By the definition of $\nu_D$, we can find $\varphi \in {\rm PSH} (V, \omega )$ such that $\varphi \in {\rm PSH}_0(D, \omega|_D)$ and $\Phi \geq (\alpha + \epsilon)^{-1} \Psi + \varphi \,\, {\rm on}\,\, V$ for each $\Phi \in \mathcal{U}$.
Directly, we have
\begin{equation*}
\label{computation}
\int_W e^{-\alpha \Phi } d \mu_0 \leq \int_W e^{- \alpha (\alpha +\epsilon)^{-1}\Psi - \alpha \varphi } d\mu_0 \leq \int_W e^{-\alpha \varphi} d\mu_\gamma.
\end{equation*}
Here, we write $\gamma = \alpha (\alpha +\epsilon)^{-1} <1$.
By the definition of $A$, we have $\alpha < \alpha_{\omega |_D} (d\lambda ^\prime)$.
Since $\varphi |_D \in {\rm PSH }_0 (D , \omega|_D) $, it follows from Theorem \ref{Tian} that the integral $\int_D e^{-\alpha \varphi} d \lambda^\prime $ is uniformly bounded.
By taking a finite covering of $D$ and repeating the proof of Theorem \ref{L1}, we can find a sufficiently small neighborhood $W$ of $D$ such that there exists an uniform constant $C>0$ independent of $\varphi$, (so, $\Phi$) such that $ \int_W e^{-\alpha \varphi} d\mu_\gamma \leq C \int_D e^{-\alpha \varphi} d \lambda^\prime$.
Thus, we have finished the proof of Theorem \ref{alpha invariant}.
\sq

\begin{rem}
By taking a suitable $\gamma \in (0,1)$, Theorem \ref{alpha invariant} can be applied to the integrability threshold for the measure defined by a complete {\kah} metric which does not has Poincar\'{e} type asymptotic behavior.
In particular, Sz\'{e}kelyhidi \cite[\S 3.1]{Sz} shows that there is a {\kah} manifold with an extremal {\kah} metrics with the asymptotic
$$\frac{i dz \wedge d \overline{z}}{|z|^2 (\log |z|^{-2})^{\frac{3}{2}}}.$$
By the direct computation, we can find a corresponding {\kah} potential $ \Phi_{1/2}$ with the asymptotic $(-\log |z|^2)^{\frac{1}{2}}$ (which also has zero Lelong number).
However, the computation in Example \ref{example 2} implies that we have $\int \exp (-\alpha \Phi_{1/2}) d\mu_{1/2} = \infty$ for any $\alpha >0$, so $\alpha ( \{ \Phi_{1/2} \} , \mu_{1/2}) = 0$.
Thus, the positivity of integrability thresholds in this case is more complicated.
\end{rem}


\section{Entropy compactness theorem}

In this section, we prove a certain compactness result of the relative entropy.
Firstly, we recall quickly the definition of the finite energy space denoted by $\mathcal{E}^1 (X,\omega)$ and the $d_1$-metric.

\subsection{Finite energy space and $d_1$-metric}

In this subsection, we recall the definition of the finite energy space $(\mathcal{E}^1 (X, \omega), d_1)$ which is a complete geodesic metric space, i.e., $(\mathcal{E}^1 (X, \omega), d_1)$ is a complete metric space such that any two points can be connected by a geodesic path.
We only mention the results in \cite{Da1} and \cite[Section 3 and 4]{Da}  without proof in order to prove Theorem \ref{ent cpt}.

Let $\mathcal{H} = \mathcal{H} (X ,\omega)$ be the space of {\kah} metrics cohomologous to $ \omega $, i.e., $\mathcal{H} = \mathcal{H} (X ,\omega):=  \{ \varphi \in C^\infty (X)\, |\, \omega_\varphi = \omega + \dol \varphi >0 \, \}.$
The space $\mathcal{H}$ has an infinite dimensional Riemannian manifold structure defined by the $L^2$-inner product called the Mabuchi metric \cite{Do, Se, Ma}.
Lempert-Vivas \cite[Theorem 1.1]{LV} proved that in general, $\mathcal{H}$ is not geodesically complete, i.e., there is a {\kah} manifold with two {\kah} metrics such that  there is no smooth geodesic in $\mathcal{H}$ connecting them.
We recall the finite energy class $\mathcal{E}^1 (X ,\omega)$ in $\mathcal{E} (X ,\omega)$ and we see later that the space $\mathcal{E}^1 (X ,\omega)$ is the metric completion of $\mathcal{H}$.

\begin{definition}{\rm (\cite{BEGZ,GZ2})}
We define the set of $\omega $-plurisubharmonic functions of {\it finite energy} by
$$
\mathcal{E}^1 (X, \omega):= \left\{  \varphi  \in \mathcal{E} (X, \omega) \,\,   \middle| \,\, \int_X |\varphi| \omega_{\varphi}^n < \infty  \right\}.
$$

\end{definition}

\begin{exam}
We can easily show that $\Psi =- \log \log \Vert s_D \Vert^{-2}_{h} \in \mathcal{E}^1 (X , \omega)$ by the direct computation.
So, the space of potential functions of {\kah} metrics of Poincar\'{e} type is included in $\mathcal{E}^1 (X, \omega)$ (see definition of {\kah} metrics of Poincar\'{e} type in \cite{Au}).
\end{exam}

\begin{exam}
On the other hand, we can find a $\omega$-plurisubharmonic function with full mass which is not in the finite energy space.
Indeed, for sufficiently small $\epsilon > 0$, one can show that $\Psi_{\beta}:= -\epsilon ( \log \Vert s_D \Vert^{-2}_{h})^\beta \in \mathcal{E}(X , \omega ) \setminus \mathcal{E}^1 (X , \omega)$ for $\beta \in [1/2,1)$ by the direct computation.
\end{exam}

\begin{definition}
For $\varphi_0 ,  \varphi_1 \in \mathcal{H}$, we define {\it the $d_1$-metric} between $\varphi_0$ and $\varphi_1$ by
\begin{equation}
\label{geodesic}
d_1 (\varphi_0 , \varphi_1):= \inf \left\{  \int_0^1 dt \int_X \left| \dot{\varphi} \right| \omega_{ \varphi_t }^{n} \, \middle|\,  \{ \varphi_t \}_{0 \leq t \leq 1} \subset \mathcal{H}  {\rm \,\, is \,\, smooth \,\, in \,\,} t \right\}.
\end{equation}
\end{definition}

\begin{rem}
X.X. Chen \cite{Chen} proved that there exists a geodesic $\{ \varphi_t \}_{0\leq t \leq 1}$ connecting $\varphi_0, \varphi_1 \in \mathcal{H}$ by solving the homogeneous complex {\ma} equation.
(Note that $\varphi_t $ is not smooth in general.)
This geodesic is called the $C^{1 , \overline{1}}$-geodesic and defined by the Mabuchi metric.
Darvas \cite[Theorem 3.6]{Da} proved that the $C^{1 , \overline{1}}$-geodesic minimizes the length between any two points (with respect to $d_p$-metric for $p \geq 1$).
For more details, see \cite[Section 3]{Da}.
\end{rem}

Darvas proved the following result which says that the finite energy space $\mathcal{E}^1 (X ,\omega)$ is compatible with the metric space structure of $(\mathcal{H}, d_1)$.

\begin{thm}{\rm ($p=1$ for \cite
[Theorem 2]{Da1} and \cite[Theorem 3.36]{Da})}
The finite energy space $(\mathcal{E}^1 (X, \omega) , d_1)$ is the metric completion of $(\mathcal{H}, d_1)$.
Moreover, $(\mathcal{E}^1 (X, \omega) , d_1)$ is a complete geodesic metric space.

\end{thm}
As we will see later, for $\varphi_j , \varphi \in \mathcal{E}^1 (X,\omega)$, if $d_1 (\varphi_j , \varphi) \to 0$ then $\int_X |\varphi_j - \varphi| \omega^n \to 0$ as $j\to \infty$.
However, the converse does not hold in general, so the convergence in the sense of $d_1$-metric is strictly stronger than the convergence in the sense of $L^1 (\omega^n)$-topology.
In order to describe the topology of $(\mathcal{E}^1 (X, \omega),d_1)$, the following energy functional plays a fundamental role.

\begin{definition}{\rm (see \cite[Section 2]{BEGZ})}
We define {\it the Monge-Amp$\grave{e}$re energy} (which is also called {\it the Aubin-Mabuchi}, or {\it Aubin-Yau energy}) by
$$
I(\varphi):= \frac{1}{V(n+1)} \sum_{j=0}^{n} \int_X \varphi \omega_{\varphi}^j \wedge \omega^{n-j}, \,\,\, \varphi \in {\rm PSH} (X , \omega) \cap L^\infty (X).
$$
For $\varphi \in {\rm PSH}(X, \omega)$, we define
$$
I(\varphi):= \lim_{k \to \infty } I(\varphi_k) \in \mathbb{R} \cup \{ -\infty \},
$$
where $\varphi_k$ denotes the canonical cutoff of $\varphi$.
\end{definition}

We mention the upper semicontinuity of $I$ with respect to $L^1 (\omega^n)$-topology.

\begin{prop}{\rm (see \cite[Corollary 4.14]{Da})}
\label{equivalence}
For $u,u_j \in \mathcal{E}^1 (X,\omega)$, the Monge-Amp$\grave{e}$re energy $I: \mathcal{E}^1 (X,\omega) \to \mathbb{R} \cup \{-\infty\}$ is upper semicontinuous for $L^1 (\omega^n)$-topology, i.e., if $\varphi_j \to \varphi$ in $L^1(\omega^n)$, then we have
$$
\limsup_{j \to \infty} I(\varphi_j) \leq I(\varphi).
$$
\end{prop}

The convergence in the sense of $d_1$-metric is characterized by the $L^1 (\omega^n)$-convergence and the energy convergence as follows.

\begin{thm}{\rm (\cite[Proposition 5.10]{Da1}, \cite[Theorem 3.46]{Da})}
\label{d1 convergence}
For $\varphi_j, \varphi \in \mathcal{E}^1 (X ,\omega)$, $d_1 (\varphi_j , \varphi ) \to 0$ if and only if $\int_X |\varphi_j  - \varphi| \omega^n \to 0$ and $I(\varphi_j) \to I( \varphi)$ as $j\to \infty$.
\end{thm}

In order to prove Theorem \ref{ent cpt}, we recall important inequalities in \cite[Section 3 and 4]{Da} and \cite{Da1} between $d_1, I, \sup_X$ and $L^1$-norm.

\begin{prop}{\rm (\cite[Proposition 3.40]{Da})}
\label{ineq}
For $\varphi \in {\rm PSH} (X , \omega) $, $\varphi \in \mathcal{E}^1 (X , \theta_X)$ if and only if $I(\varphi) > -\infty$.
Moreover, for $\varphi_0, \varphi_1 \in \mathcal{E}^1 (X , \theta_X)$, we have
\begin{eqnarray}
&&\label{ineq11} \hspace{-50pt}| I(\varphi_0) - I(\varphi_1)| \leq d_1  (\varphi_0, \varphi_1), \\
\frac{1}{V} \int_X (\varphi_0 - \varphi_1 )\omega_{ \varphi_0}^{n} &\leq& \label{ineq12} I(\varphi_0) -I(\varphi_1) \, \leq \, \frac{1}{V} \int_X (\varphi_0 - \varphi_1 )\omega_{ \varphi_1}^{n} .
\end{eqnarray}
\end{prop}

\begin{lemma}{\rm (\cite[Lemma 3.45]{Da})}
\label{ineq2}
There exists a constant $C>0$ depending only on $X$ and $\omega$, such that
$$
\frac{1}{V} \int_X \varphi \omega^n \leq \sup_X \varphi \leq \frac{1}{V} \int_X \varphi \omega^n + C
$$
for all $\varphi \in {\rm PSH} (X, \omega).$
\end{lemma}

\begin{thm}{\rm ($p=1$ for \cite[Theorem 3]{Da1} and \cite[Theorem 3.32]{Da})}
\label{d1 L1}
For $\varphi_0, \varphi_1 \in \mathcal{E}^1 (X ,\omega)$, we have
$$
d_1 ( \varphi_0, \varphi_1 ) \leq \frac{1}{V} \int_X |\varphi_0 - \varphi_1 | \omega_{ \varphi_0}^{n} + \frac{1}{V} \int_X | \varphi_0 - \varphi_1 | \omega_{ \varphi_0}^{n} \leq 2^{2n +6} d_1 ( \varphi_0, \varphi_1 ).
$$
\end{thm}

Finally, we state the $L^1(\omega^n)$-compactness result for some bounded class in $\mathcal{E}^1 (X ,\omega)$.

\begin{thm}{\rm (\cite[Lemma 2.6]{BBGZ}, see also \cite[Lemma 4.13]{Da})}
\label{d1 bounded}
For $C_1 , C_2 \in \mathbb{R}$, the subset defined by
$$
\left\{ \varphi \in \mathcal{E}^1 (X , \omega) \,\, \middle|  \,\, C_1 \leq I(\varphi) \leq \sup_X \varphi \leq C_2  \right\}
$$
is compact with respect to weak $L^1 (\omega^n)$-topology.
\end{thm}

\subsection{Proof of Theorem \ref{ent cpt}}

In this subsection, we discuss the entropy compactness theorem by using the integrability threshold in Section 4.

\begin{definition}
For two probability measures $d\nu_1, d\nu_2$ on $X$, we define the {\it relative entropy of $d\nu_2$ for $d\nu_1$} by 
\begin{equation*}
{\rm Ent}_{d\nu_1} \left( d\nu_2 \right):= \int_{X} \log \left( \frac{d \nu_2}{d\nu_1} \right) d d\nu_2
\end{equation*}
if $d\nu_2$ is absolutely continuous to $d\nu_1$ or ${\rm Ent}_{d\nu_1} \left( d\nu_2 \right) = + \infty$ otherwise.
Here, the measurable function $\frac{ d\nu_2}{d\nu_1}$ is the Radon-Nikod\'{y}m derivative of $d\nu_2$ for $d\nu_1$.
\end{definition}

\begin{rem}
It is well-known that the relative entropy satisfies ${\rm Ent}_{ d\nu_1} (d\nu_2) \geq 0$ and the equality holds if and only if $d\nu_1 = d\nu_2$ by the Jensen inequality.
It is also well-known that the relative entropy is the Legendre transform of the functional $\phi \to \log \int_X e^{\phi} d \nu_1$.
We refer to $\cite[\S2.3, \S2.4]{BBEGZ}$ for more details of the relative entropy.
\end{rem}

\begin{rem}
The relative entropy for some measure with singularities is sometimes called the {\it log} relative entropy (or log entropy).
In particular, the log relative entropy for the measure $d \lambda_\beta$ with {\it cone} singularities with angle $2\pi \beta$ plays an important role in \cite{Zh} for the study of the log Mabuchi K-energy.
We can regard our measure $d\mu_0$ as the limit of $d \lambda_\beta$ as $\beta \to 0$ (see singular {\kah} metrics and their approximations in \cite{Gu, Ao}). 
We can call the functional ${\rm Ent}_{\mu_0}$ the log relative entropy (for angle $0$), but in this paper, we just call it the relative entropy for simplicity.

\end{rem}

\begin{rem}
The outline of the proof of Theorem \ref{ent cpt} is similar to the proof in \cite[Theorem 4.44]{Da}.
The difference between them is that our case deals with the relative entropy for the (very!) singular reference measure $\mu_\gamma$.
So, we need to show that the integrals $\int_X e^{-\alpha \Phi_j} d \mu_\gamma, \,\, \int_X e^{-\alpha \Phi} d \mu_\gamma$ are uniformly bounded by choosing a {\it suitable} constant $\alpha >0$.
Moreover, we need to show that some $L^1 (\omega^n)$-convergence sequence also converges in the sense of $L^p (d \mu_\gamma)$-topology.

\end{rem}

{\it Proof of Theorem \ref{ent cpt}.}
As in the proof of Theorem \ref{alpha invariant}, we assume that $\gamma =0$.
By scaling, we assume that $V = \int_X \omega^n = 1$.
Since $\mathcal{U} = \{ \Phi_j \}_j$ is $d_1$-bounded, i.e., $d_1 (0 ,\Phi_j) <K$ for any $j$, the inequality (\ref{ineq11}) in Proposition \ref{ineq}, Lemma \ref{ineq2} and Theorem \ref{d1 L1} implies that $| \sup_X (\Phi_j) |$ and $I(\Phi_j)$ are bounded.
Thus, Lemma \ref{d1 bounded} implies that there exists $\Phi \in \mathcal{E}^{1} (X ,\omega) $ with $\Vert  \Phi_j  - \Phi \Vert_{L^1 (\omega^{n})} \to 0$ by taking a subsequence.
So, it follows from Theorem \ref{d1 convergence} that we only need to prove $\lim_{j \to \infty} I(\Phi_j) = I(\Phi)$.
From Proposition \ref{equivalence}, we know that $\limsup_{j \to \infty} I(\Phi_j) \leq I(\Phi)$.
By the second inequality (\ref{ineq12}) in Proposition \ref{ineq}, we have
$$
I(\Phi) \leq  I(\Phi_j) +  \int_{X} |\Phi - \Phi_j| \omega_{\Phi_j}^n,
$$
so we can show that $\liminf_{j \to \infty} I(\Phi_j) \geq I(\Phi)$ if the second term in the inequality above goes to $0$ when $j\to \infty$.
Thus, it suffices to prove that $\int_{X} |\Phi - \Phi_j| \omega_{\Phi_j}^n \to 0$ as $j \to \infty$.

We consider increasing and convex functions $a, b: \mathbb{R}_{\geq 0} \to \mathbb{R}_{\geq 0}$ defined by
$$
a(s):= (s+1) \log (s+1) -s , \,\,\,\,\,\,\, b(t):=  e^t -t-1.
$$
Note that $b$ is the Legendre conjugate of $a$, i.e., $b = a^*$.
Since $a(s) \leq s\log s + O(1)$ and we assume that ${\rm Ent}_{d \mu_0} (V^{-1}\omega_{\Phi_j}^{n} )<K $ for any $j$, we can find $K_r >0$ for any $r>0$ such that
$$
\left\Vert r^{-1} \frac{ \omega_{\Phi_j}^{n}}{d\mu_{0}}  \right\Vert_{L^a(d \mu_0)} < K_r .
$$
Here, the symbol $\Vert \cdot \Vert_{L^a(d \mu_0)}$ denotes the Orlicz norm with respect to the function $a$ (see \cite[Section 2]{BBEGZ} and \cite[Section 1]{Da}).
Since $b(t) \leq t^2 e^t$ for any $r>0$, the {\hol} inequality for a pair of $a$ and $b$ (see \cite[Proposition 2.15]{BBEGZ} and \cite[Proposition 1.3]{Da}) implies that
\begin{eqnarray*}
\int_{X} |\Phi - \Phi_j| \omega_{\Phi_j}^n 
&=& \int_{X} |\Phi - \Phi_j| \frac{\omega_{\Phi_j}^n}{d\mu_{0}} d\mu_{0}  \\
&\leq&  2 \left\Vert r^{-1} \frac{ \omega_{\Phi_j}^{n}}{d\mu_{0}}  \right\Vert_{L^a(d \mu_0)}  \left\Vert   r|\Phi - \Phi_j|  \right\Vert_{L^b(d \mu_0)}
  \\
&\leq& 2K_r \left\Vert   r|\Phi - \Phi_j|  \right\Vert_{L^b(d \mu_0)}
\end{eqnarray*}

Note that
$$A:= \min \left\{  \,  \alpha_{\omega |_D} (d \lambda^\prime) , \,  \frac{1}{\nu_D (\mathcal{U})} \right\}
$$
is positive by Theorem \ref{Tian} and the assumption of Theorem \ref{ent cpt}.
So, we can find sufficiently small $r>0$ such that $\alpha:= 3r < A$.
By using the {\hol} inequality, we obtain
\begin{eqnarray*}
\int_X b\left(  r|\Phi - \Phi_j| \right) d\mu_{0}   &\leq& r^2 \int_{X} |\Phi - \Phi_j|^2 \exp(r|\Phi - \Phi_j| )d\mu_{0} \\
 &\leq& r^2 \Vert \Phi - \Phi_j \Vert_{L^6 (d\mu_{0} )}^{2} \Vert \exp (-r \Phi_j) \Vert_{L^3 (d\mu_{0} )}  \Vert \exp (-r\Phi) \Vert_{L^3 (d\mu_{0} )} .
\end{eqnarray*}
By using Corollary \ref{cpt ver}, we can find a constant $C >0$ independent of $j$ such that $\int_X e^{- \alpha \Phi_j} d \mu_0 \leq  C
$ for all $j$.
Moreover, Fatou's lemma implies that we have $ \int_X e^{-\alpha \Phi} d \mu_0 \leq C$.
By Theorem \ref{local Lp}, we can show that $\Vert \Phi - \Phi_j \Vert_{L^6 (d\mu_{0} )} \to 0 \,\,\, (j \to \infty)$, so it follows that $\int_{X} |\Phi - \Phi_j| \omega_{\Phi_j}^n \to 0$.
Thus, we have finished the proof of Theorem \ref{ent cpt}.
\sq

\begin{exam}(Comparison with Theorem 2.8 in \cite{BDL1})
\label{diverge}
For sufficiently small $\delta \geq 0$, we set a function
$$
\Psi_{\frac{1}{2} - \delta} := - \epsilon (\log \Vert s_D \Vert^{-2}_{h})^{\frac{1}{2} - \delta}
$$
for a fixed sufficiently small constant $\epsilon > 0$.
Note that $\Psi_{\frac{1}{2} - \delta} \in \mathcal{E}^1 (X,\omega)$ for $\delta > 0$.
By the direct computation, we obtain a uniform constant $C > 0$ such that
$$
\left| \sup_X \Psi_{\frac{1}{2} - \delta} \right| <C ,\,\,\,  {\rm Ent }_{d \mu_0}\left(\omega_{\Psi_{\frac{1}{2} - \delta}}^{n} \right) <C
$$
for any $\delta \geq 0$, i.e., the sequence $\{\Psi_{\frac{1}{2} - \delta} \}_{\delta > 0} \subset \mathcal{E}^1 (X,\omega)$ satisfies the assumption of Theorem 2.8 in \cite{BDL1}.
However, if we set $\delta_{j} : = 1/j$, the sequence $\{\Psi_{\frac{1}{2} - \delta_j} \}_{j} $ does not have a $d_1$-convergent subsequence, because $\Psi_{\frac{1}{2}} \notin \mathcal{E}^1 (X,\omega) $.
Note that this is an example of quasi-plurisubharmonic functions with finite entropy (with respect to the relative entropy for the reference measure ${\mu_0}$!) which is not of finite energy (see \cite[Theorem 2.17]{BBEGZ} and \cite[Theorem A]{DGL}).
\end{exam}

\medskip

{\it Proof of Corollary \ref{DNL}.}
It follows from (b) of Theorem 2 in \cite{DL}, there exists $A_1, A_2 >0$ depending only on $B$ such that $\Phi_j \geq A_1 \Psi + A_2$ for any $j$.
Thus, we can find $r>0$ such that $r \Psi \leq A_2$ on the neighborhood $V$ in the definition of the log-log threshold $\nu_D$, so we obtain $\nu_D (\Phi_j) \leq A_1+r$ for any $j$.
We write $f := \frac{1}{ \Vert s_D \Vert^{2} (\log \Vert s_D \Vert^{-2})^{2}}$.
So, it follows from the inequality $\omega_{\Phi_j}^{n}  \leq B f \omega^n \approx B d\mu_0$ that there exists some constant $C_B > 0$ depending only on $B$ such that ${\rm Ent}_{d\mu_0} (V^{-1}\omega_{\Phi_j}^{n} ) < C_B$.
Thus, we complete the proof by Theorem \ref{ent cpt}.\sq




\bigskip
\address{
National Institute of Technology (KOSEN),\\
Wakayama College,\\
77, Nojima, Nada-chou, Gobo-shi\\
Wakayama, 644-0023\\
Japan
}
{aoi@wakayama-nct.ac.jp \,\, {\it or} \,\, takahiro.aoi.math@gmail.com }

\end{document}